\documentclass[12pt]{amsart}
\usepackage{graphicx,amssymb,latexsym,amsfonts,amsmath,amsthm,mathrsfs}

\usepackage{color,xcolor}
\usepackage{graphicx}
\usepackage{subfigure}
\usepackage{listing}
\usepackage[pagebackref]{hyperref}
\usepackage{url}
\usepackage{enumerate}
\usepackage{multirow}
\advance\hoffset by -0.9 truecm

\newtheorem{theorem}{Theorem}[section]
\newtheorem{proposition}[theorem]{Proposition}
\newtheorem{corollary}[theorem]{Corollary}
\newtheorem{lemma}[theorem]{Lemma}

\theoremstyle{definition}
\newtheorem{construction}[theorem]{Construction}
\newtheorem{example}[theorem]{Example}
\newtheorem{definition}[theorem]{Definition}

\theoremstyle{remark}

\def\calB{{\mathcal B}}

\def\Cos{\mathrm{Cos}}

\def\l{\langle}
\def\r{\rangle}
\def\ov{\overline}

\def\a{\alpha}
\def\b{\beta}
\def\g{\gamma}
\def\o{\omega}

\def\Sig{{\it\Sigma}}
\def\Ga{{\it\Gamma}}
\def\Del{{\it\Delta}}

\def\Sym{{\rm Sym}}

\def\A{{\rm A}}
\def\S{{\rm S}}

\def\C{{\bf C}}

\def\val{\mathrm{val}}

\def\D{{\rm D}}

\def\Aut{\mathrm{Aut}}

\def\ZZ{{\mathbb Z}}
\def\K{\mathbf{K}}

\begin{document}

\title{Covers and pseudocovers of symmetric graphs}

\author{Cai Heng Li}

\address{SUSTech International Center for Mathematics, and Department of Mathematics, Southern University of Science and Technology\\
Shenzhen 518055, Guangdong\\
P. R. China}
\email{lich@sustech.edu.cn {\text{\rm(Li)}}}

\author{Yan Zhou Zhu}

\address{Department of Mathematics, Southern University of Science and Technology\\
Shenzhen 518055, Guangdong\\
P. R. China}
\email{zhuyz@mail.sustech.edu.cn {\text{\rm(Zhu)}}}


\begin{abstract}
We introduce the concept of {\it pseudocover}, which is a counterpart of {\it cover}, for symmetric graphs.
The only known example of pseudocovers of symmetric graphs so far was given by Praeger, Zhou and the first-named author a decade ago, which seems technical and hard to extend to obtain more examples.
In this paper, we present a criterion for a symmetric extender of a symmetric graph to be a pseudocover, and then apply it to produce various examples of pseudocovers, including
(1) with a single exception, each Praeger-Xu's graph is a pseudocover of a wreath graph;
(2) each connected tetravalent symmetric graph with vertex stabilizer of size divisible by $32$ has connected pseudocovers.
\end{abstract}

\maketitle

\section{Introduction}\label{sec:introduction}

Denote by $\Ga=(V,E)$ a finite undirected simple graph with vertex set $V$ and edge set $E$.
The vertex set $V$ is sometime written as $V\Ga$.
An \textit{arc} of $\Ga$ is an ordered pair of adjacent vertices, and an edge corresponds to two arcs.
A graph $\Ga$ is called \textit{symmetric} if its arcs are equivalent under the automorphism group $\Aut\Ga$.
In particular, we say $\Ga$ is $G$-symmetric for if $G\leqslant\Aut(\Ga)$ is transitive on the arc set.
A symmetric graph is also called an \textit{arc-transitive} graph.
The study of symmetric graphs is one of the main topics in algebraic graph theory, refer to \cite{biggs1974Algebraic,godsil2001Algebraic} for references.

For a graph $\Ga=(V,E)$, a group of automorphisms $G\leqslant\Aut\Ga$ is a permutation group on the vertex set $V$.
If $G$ is further transitive on arcs of $\Ga$, then $G$ is a transitive permutation group on $V$.
Studying symmetric graphs is thus naturally associated with studying permutation groups $G$.
In the case where $G$ is \textit{primitive} on $V$, namely, $\Ga$ is \textit{$G$-vertex-primitive}, the well-known O'Nan-Scott Theorem~\cite{liebeck1988NanScott} of permutation group theory and the theory of finite simple groups provide powerful tools for the study of symmetric graphs which are vertex-primitive, see \cite{li2001finite,praeger1993NanScott}.

On the other hand, assume that a $G$-symmetric graph $\Ga=(V,E)$ is not $G$-vertex-primitive, namely, $G$ is imprimitive on $V$.
Let $\calB$ be a non-trivial \textit{block system} for $G$ acting on $V$, that is, $1<|\calB|<|V|$.
Then $G$ induces a transitive permutation group $G^\calB$, and $G/G_{(\calB)}\cong G^\calB$, which is called a {\it quotient permutation group} of $G$ on $V$, where $G_{(\calB)}$ is the kernel of $G$ acting on $\calB$.
If $G_{(\calB)}=1$, then $G\cong G^\calB$ is faithful on $\mathcal{B}$.
In the quotient action, the point in $\calB$ corresponding to the block $B$ containing $\b$ is denoted by $\ov\b$, so that $\b\in B=\ov\b$.
Clearly, one block $\ov\b$ corresponds to $|B|$ choices for $\b$.
Corresponding to a quotient group, a $G$-symmetric graph $\Ga$ has a {\it quotient graph} $\Ga_\calB$ (with respect to $G$), which has vertex set $\calB$ such that for any two vertices $\ov\a,\ov\b\in\calB$,
\[\mbox{$(\ov\a,\ov\b)$ is an arc of $\Ga_\calB$ if and only if $(\a,\b)$ is an arc of $\Ga$ for some $\a\in\ov\a$ and $\b\in\ov\b$.}\]
We usually call the original graph $\Ga$ an {\it extender} of the quotient graph $\Ga_\mathcal{B}$, and call the block $\overline{\alpha}$ the \textit{image} of $\alpha$.

The quotient graph $\Ga_\calB$ is a smaller symmetric graph than the original graph $\Ga$ in the sense that $|V\Ga_\calB|<|V\Ga|$.
Naturally, one may wish to understand $\Ga$ through characterizing the smaller graph $\Ga_\calB$ and the relation between $\Ga$ and $\Ga_\calB$.
For two blocks $\ov\a$ and $\ov\b$ in $\calB$ which are adjacent vertices in $\Ga_\calB$, let $\Ga[\overline{\alpha},\overline{\beta}]$ be the subgraph of $\Ga$ induced on the subset of vertices $\overline{\alpha}\cup \overline{\beta}$ of $V\Ga$.
If $\Ga$ is $G$-symmetric, then the induced subgraph $\Ga[\overline{\alpha},\overline{\beta}]$ is independent of the choice of $(\ov\a,\ov\b)$.
The structure of $\Ga[\overline{\alpha},\overline{\beta}]$ and the valencies $\val(\Ga)$ and $\val(\Ga_\calB)$ are obviously important parameters for understanding the relation between $\Ga$ and $\Ga_\calB$, which leads to the following concepts.

\begin{definition}\label{def:pscover}
	Let $\Ga$ be a $G$-symmetric graph, and $\Ga_\mathcal{B}$ a quotient graph of $\Ga$, where $\calB$ is a block system of $G$ acting on $V\Ga$.
	Pick an arc $(\alpha,\beta)$ of $\Ga$.
	\begin{enumerate}[{\rm (A)}]
		\item If the induced subgraph $\Ga[\overline{\alpha},\overline{\beta}]$ is a regular graph, then $\Ga$ is called a \textit{multicover} of $\Ga_\mathcal{B}$; so that $\val(\Ga)$ is divisible by $\val(\Ga_\mathcal{B})$.

		\item If the induced subgraph $\Ga[\overline{\alpha},\overline{\beta}]$ is a perfect matching, then $\Ga$ is called a \textit{cover} of $\Ga_\mathcal{B}$; so that $\val(\Ga)=\val(\Ga_\mathcal{B})$.

		\item  If $\val(\Ga)=\val(\Ga_\mathcal{B})$ and $\Ga$ is not a cover of $\Ga_\mathcal{B}$, then $\Ga$ is called a \textit{pseudocover} of $\Ga_\mathcal{B}$.
	\end{enumerate}
\end{definition}

Constructing and characterizing symmetric covers of given symmetric graphs is an important approach for studying symmetric graphs, refer to \cite{brouwer1989Distanceregular,du1998arctransitive,du20052arctransitive}.
As noticed above, if a symmetric graph $\Ga$ is a cover of $\Ga_\mathcal{B}$, then $\Ga$ and $\Ga_\calB$ have the same valency.
Whether the converse statement is true or not had been an open problem until the first example of pseudocovers was discovered by Praeger, Zhou and the first-named author in~\cite{li2010Imprimitive}, see Example~\ref{exam:li4pscover}.
The presentation given in~\cite[Construction 3.5]{li2010Imprimitive} seems technical and hard to extend to obtain more examples.

{\bf Convention:}\ 
For a $G$-symmetric graph $\Ga$ where $G\leqslant\Aut\Ga$, the group $G$ induces an arc-transitive action on each quotient graph $\Ga_\calB$.
However, $G$ is not necessarily faithful on the vertex set $V\Ga_\calB$.
For convenience, we shall also call $\Ga_\calB$ a {\it $G$-symmetric} graph, and say $\Ga$ is a $G$-symmetric extender of $\Ga_\mathcal{B}$ if $\mathcal{B}$ is a block system of $G$.
With this convention, $G_\a$ is a subgroup of $G_{\ov\a}$ and acts on $\Ga_\calB(\ov\a)$ naturally, where $G_{\ov\a}$ is the subgroup of $G$ stabilizes the subset $\ov\a$.

We first present a criterion for deciding an extender to be a pseudocover.

\begin{theorem}\label{thm:pscover}
	Let $\Ga$ be a $G$-symmetric extender of $\Ga_\mathcal{B}$ such that $\val(\Ga)=\val(\Ga_{\mathcal{B}})$.
	Let $(\alpha,\beta)$ be an arc of $\Ga$, and $(\overline{\alpha},\overline{\beta})$ the corresponding image in $\Ga_\mathcal{B}$.
	Then the followings are equivalent:
	\begin{enumerate}[{\rm(a)}]
	\item $\Ga$ is a $G$-symmetric pseudocover of $\Ga_{\mathcal{B}}$;

	\item $G_\a$ is an intransitive subgroup of $G_{\ov\a}$ on $\Ga_\mathcal{B}(\overline{\alpha})$;

	\item $G_{\overline{\alpha}}\neq G_\alpha G_{\ov{\alpha}\ov{\beta}}$;

	\item $G_{\alpha\overline{\beta}}$, $G_{\overline{\alpha}\beta}$ and $G_{\alpha\beta}$ are pairwise different.
	\end{enumerate}
\end{theorem}

This theorem explores a key property of pseudocovers, and provides us with a tool for constructing covers and pseudocovers of symmetric graphs, which we will apply to study several important families of symmetric graphs.

In 1989, Praeger and Xu \cite{praeger1989characterization} explicitly characterized the class of $G$-symmetric graphs of order $rp^s$ and valency $2p$ with $p$ prime and $r\geqslant3$ such that $G$ has an elementary abelian normal $p$-subgroup that is non-semiregular on vertices, which we call {\it Praeger-Xu's graphs}.
See Definition~\ref{def:Praeger-Xu} for a definition of these graphs in terms of coset graphs.
We remark that the Praeger-Xu's graph with $s=1$, namely the graph is of order $rp$, is isomorphic to the so-called {\it wreath graph} $\mathbf{W}(r,p)=\C_r[\ov\K_p]$.

The next result gives a characterization of Praeger-Xu's graphs in the language of pseudocovers.

\begin{corollary}\label{cor:PX-graphs}
	With a single exception, each Praeger-Xu's graph is a symmetric pseudocover of $\mathbf{W}(r,p)$, for some integer $r\geqslant 3$ and prime $p$.
\end{corollary}

The final theorem of this paper shows that `most' tetravalent symmetric graphs have connected pseudocovers.

\begin{theorem}\label{thm:pscover4}
Each connected tetravalent $G$-symmetric graph with vertex stabilizer of order divisible by $32$ has a connected $G$-symmetric pseudocover.
\end{theorem}

\section{Coset graphs and extenders}\label{sec:pre}

A well-known important notion for studying symmetric graphs is the \textit{coset graph representation}, which we describe below.

Let $\Ga=(V,E)$ be a $G$-symmetric graph, and fix an edge $\{\a,\b\}\in E$.
Then the vertex set $V$ may be identified with the set of right cosets $[G:G_\a]=\{G_\a x\mid x\in G\}$, and then the action of $G$ on $V$ is equivalent to the action of $G$ on $[G:G_\a]$ by right multiplication of elements in $G$.
Further, since $G$ is symmetric on $\Ga$, there exists an element $g\in G$ such that $(\alpha,\beta)^g=(\beta,\alpha)$.
Clearly, $g^2\in G_\a$ and $g$ can be chosen as a $2$-element, namely, an element of order $2^m$ for some integer $m$.
Then, in terms of cosets, the right coset $G_\a x$ corresponds to the vertex $\alpha^x\in V$.
In particular, $G_\a$ and $G_\alpha g$ correspond to the vertices $\a$ and $\beta$, respectively.
Thus the neighborhood $\Ga(\a)=\{G_\a gh\mid h\in G_\a\}$, consisting right cosets of $G_\a$ contained in the double coset $G_\a gG_\a$.
It implies that $G_\a x$ and $G_\a y$ are adjacent if and only if $yx^{-1}\in G_\a gG_\a$.

Conversely, for an abstract group $G$, a subgroup $H<G$ and an element $g\in G$ such that $g^2\in H$, one can define a graph $\Ga=(V,E)$, called a \textit{coset graph} and denoted by $\Cos(G,H,HgH)$, where
\[\begin{array}{l}
	V=[G:H],\\
	E=\bigl\{\{Hx,Hy\}\mid yx^{-1}\in HgH\bigr\}.
\end{array}\]
We notice that, with this definition, $G$ is not necessarily faithful on the vertex set $V=[G:H]$.
The quotient permutation group $G^V$ is a subgroup of $\Aut\Ga$ by the coset action of $G$ on $[G:H]$.
As usual, the largest normal subgroup of $G$ which is contained in $H$ is called the {\it core} of $H$ in $G$.
Then the kernel of $G$ acting on $[G:H]$ equals the core of $H$ in $G$.
We will say $H$ is \textit{core-free} in $G$ if the core of $H$ in $G$ is trivial.

\begin{lemma}\label{lem:coset-graph}
	Suppose that $\Sig$ is a symmetric graph and $G$ is a group acting on the arc set of $\Sig$ transitively.
	Then $\Sig=\mathrm{Cos}\bigl(G,H,H g H\bigr)$, where $H=G_\a$ and $g^2 \in H$ such that $\alpha^g$ is adjacent to $\alpha$.
	Moreover, the following statements hold:
	\begin{enumerate}[{\rm(a)}]
		\item $G$ is a subgroup of $\Aut\Sig$ if and only if $H$ is core-free in $G$;
		\item $\Sig$ has valency equal to $\bigl|H:H\cap H^g\bigr|$;
		\item the action of $H=G_\alpha$ on the neighborhood $\Sig(\a)$ is equivalent to $H$ acting on $[H:H\cap H^g]$ by right multiplication;
		\item $\Sig$ is connected if and only if $G=\langle H,g\rangle$.
	\end{enumerate}
\end{lemma}

Let $\Sig=\Cos(G,H,HgH)$ and $\Ga=\Cos(G,L,LgL)$ be symmetric graphs, and assume that $L<H$.
Then $[H:L]=\{Lh:h\in H\}$ is a block, and $[G:H]$ admits a block system $\mathcal{B}$ of $[G:L]$.
Hence $\Ga$ is an extender of $\Sig$.
Conversely, suppose that $\Ga$ is a $G$-symmetric extender of $\Sig$.
By definition, both two $G$-symmetric graphs can be presented as coset graphs.
The following lemma identifies such coset graphs.
We remark that $G$ is not necessarily faithful on the vertex set $[G:H]$.

\begin{proposition}\label{prop:extender}
	Let $\Ga$ and $\Sig$ be symmetric graphs.
	Then $\Ga$ is a $G$-symmetric extender of $\Sig$ if and only if $\Ga=\Cos(G,L,LgL)$ and $\Sig=\Cos(G,H,H gH)$, where $L<H$ and $g\in G\setminus L$ such that $g^2\in L$.
\end{proposition}
\begin{proof}
	Suppose that $\Ga$ is a $G$-symmetric extender of $\Sig$.
	Let $(\alpha,\beta)$ be an arc of $\Ga$ and $(\overline{\alpha},\overline{\beta})$ the corresponding image of in $\Sig$.
	Set $L=G_\alpha$ and $H=G_{\overline{\alpha}}$.
	Then $L<H$.
	Since $\Ga$ is $G$-symmetric, there exists $g\in G$ such that $(\alpha,\beta)^g=(\beta,\alpha)$.
	Then $g^2\in L=G_\alpha$ and $\Ga=\mathrm{Cos}\bigl(G,L,L g L\bigr)$.
	Since $(\alpha,\beta)^g=(\beta,\alpha)$, it implies that $(\overline{\alpha},\overline{\beta})^g=(\overline{\beta},\overline{\alpha})$.
	By definition, $\Sig=\mathrm{Cos}\bigl(G,H,HgH\bigr)$.

	Conversely, suppose that $\Ga=\Cos(G,L,LgL)$ and $\Sig=\Cos(G,H,H gH)$, where $L<H$ and $g\in G\setminus L$ such that $g^2\in L$.
	Then $[G:H]$ admits a block system of $G$ acting on $[G:L]$.
	It follows that $\Sig$ is a quotient graph of $\Ga$ with respect to this block system.
\end{proof}

Suppose that $(\alpha,\beta)$ is an arc in $\Ga$ and $(\overline{\alpha},\overline{\beta})$ is the corresponding image in $\Ga_\mathcal{B}$.
Observe that the blocks $\left\{\overline\beta^x|x\in G_\alpha\right\}$, namely, the blocks with some vertices adjacent to $\alpha$, admit a block system of $G_{\alpha}$ acting on $\Ga(\alpha)$.
Hence if $\Ga$ is \textit{locally primitive}, that is, $G_\alpha$ is primitive on $\Ga(\alpha)$, then $\Ga$ is a pseudocover of $\Ga_\mathcal{B}$ only when $\Ga$ is not connected (otherwise, see Proposition~\ref{cor:disc-pdcover} and Figure~\ref{fig:disconnected}).

\begin{corollary}\label{coro:locallyprim}
	Suppose that $\Ga$ is a connected $G$-symmetric graph with $\Ga_\mathcal{B}$ as its quotient.
	If $\Ga$ is locally primitive, then $\Ga$ is a cover of $\Ga_\mathcal{B}$ if and only if $\val(\Ga)=\val(\Ga_\mathcal{B})$.

	In particular, the above statement holds when $\Ga$ has a prime valency; also when $\Ga$ is $(G,s)$-arc-transitive with $s\geqslant 2$.
\end{corollary}

Let $\Sig=\Cos(G,H,HgH)$ and $\Ga=\Cos(G,L,LgL)$ be symmetric graphs with $L<H$.
Then $\Sig=\Ga_\mathcal{B}$ with $\mathcal{B}$ admitted by $[G:H]$.
Assume that $\mathcal{B}$ consists of the orbits of a normal subgroup $N\lhd G$.
In this case, $\Sig$ is called a {\it normal quotient} of $\Ga$, and denoted by $\Ga_N$.
Let $\alpha$ and $\beta$ be the vertices in $\Ga$ correspond to $L$ and $Lg$, respectively.
Then $\overline{\alpha}$ and $\overline{\beta}$ correspond to $H$ and $Hg$ in $\Sig$, respectively.
Observe that $H=LN$, $\overline{\alpha}=\{\alpha^z=Lz\mid z\in N\}$ and $\overline{\beta}=\{\beta^z=\alpha^{gz}=Lgz\mid z\in N\}$.
Hence $\alpha^z\in\overline{\alpha}$ is adjacent to $\beta^z=\alpha^{gz}\in\overline{\beta}$ for each $z\in N$.
Therefore, the subgraph $\Ga[\overline{\alpha},\overline{\beta}]$ is regular and $\Ga$ is a multicover of $\Sig$.
This proves the following conclusion.

\begin{corollary}\label{coro:cosetquot}
	If a symmetric graph $\Sig$ is a normal quotient of a symmetric graph $\Ga$, then $\Ga$ is a multicover of $\Sig$.
\end{corollary}

Note that $\Ga$ is a cover of $\Ga_N$ if and only if $\val(\Ga)=\val(\Ga_N)$, and we say $\Ga$ is a \textit{normal cover} of $\Sig$ in this case.
We remark that a $G$-symmetric cover $\Ga$ of $\Sig$ for $G\leqslant\Aut\Sig$ might be a normal cover with respect to $\Aut\Ga$, but not with respect to $G$.

\begin{example}\label{exam:cube}
	{\rm
	Let $\Sig=\K_4$ with vertex set $\Omega$, and $G\leqslant\Sym(\Omega)=\S_4$ be symmetric on $\Sig$.
	Then $(G,G_\o)=(\A_4,\ZZ_3)$ or $(\S_n,\S_3)$.
	Since $\val(\Sig)=3$ is a prime, $\Sig$ has no connected symmetric pseudocovers by Corollary~\ref{coro:cosetquot}.

	Suppose that $\Ga=\Cos(G,L,LgL)$ is a proper $G$-symmetric cover of $\Sig$.
	Then $G_\o$ has a proper subgroup $L$ of order divisible by $3$, and so $(G,G_\o)=(\S_4,\S_3)$.
	It follows that $L=\ZZ_3$, and $\Ga$ is a connected $\S_4$-symmetric cubic graph with $4\times 2=8$ vertices.
	Thus $\Ga$ is the cube graph, and $\Ga$ is a normal cover $\Sig$ with respect to $\Aut\Ga\cong\mathbb{Z}_2\times\S_4$, but not with respect to $G$.
	\qed
	}
\end{example}

For pseudocovers of symmetric graphs, the extremal case $L=H\cap H^g$ is an interesting case, which we study now.

Let $\Sig=\Cos(G,H,HgH)$, and let $L=H\cap H^g$.
Then $\Ga:=\Cos(G,L,LgL)$ is an extender of $\Sig$.
Since $|L:L\cap L^g|=1$, the graph $\Ga$ is disconnected of valency $1$ and have the same number of edges as $\Sig$.
The graph $\Ga$ is actually the so-called \textit{truncation} of $\Sig$.

Assume that $\Sig$ is of valency $m$.
Let $\Del=\Ga[\ov\K_m]$, the \textit{lexicographic product} of $\ov\K_m$ by $\Ga$, in other words, $\Del$ is the resulted graph by replacing each edge of the truncation $\Ga$ by a complete bipartite graph $\K_{m,m}$.
Clearly, $\Del$ is a disconnected pseudocover of $\Sig$. Figure~\ref{fig:disconnected} shows examples of such disconnected pseudocovers for complete graphs $\K_3$ and $\K_4$.

\begin{figure}[!ht]
	\begin{center}
		\begin{tabular}{cc}
			\includegraphics[scale=0.8]{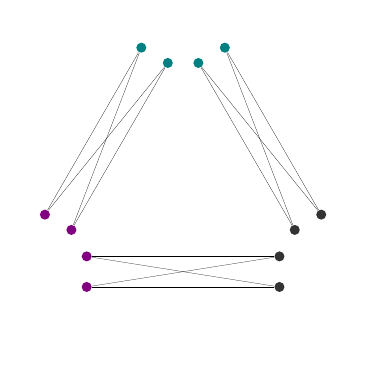}
			&
			\includegraphics[scale=0.8]{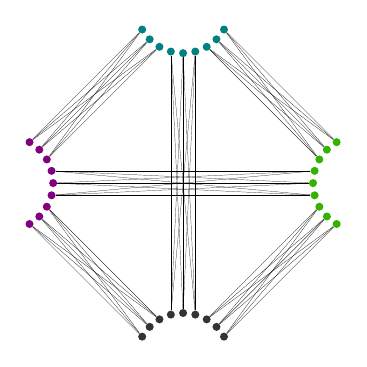}
		\end{tabular}
	\end{center}
	\caption{Disconnected pseudocovers of $\K_3$ and $\K_4$}\label{fig:disconnected}
\end{figure}

This observation leads to an interesting result.

\begin{proposition}\label{cor:disc-pdcover}
	Each symmetric graph, which is not a perfect matching, has a disconnected symmetric pseudocover.
\end{proposition}
\begin{proof}
	Let $\Sig=\Cos(G,H,HgH)$ be a $G$-symmetric graph, and let $\Ga=\Cos(G,L,LgL)$ and $\Del=\Ga[\ov\K_m]$ be the graphs defined above.
	Then $\val(\Del)=\val(\Sig)=m$ and $\val(\Ga)=1$.
	By definitions, the graph $\Ga$ is $G$-symmetric and $\Sig$ is a quotient graph of $\Ga$

	Let $V\Ga=\{\a_1,...,\a_n\}$, $\alpha=\alpha_1$ and $\beta=\alpha_2$.
	We may assume that $\alpha$ is adjacent to $\beta$ and $L=H\cap H^g$ fixes the vertex $\alpha$.
	Since $\val(\Ga)=1$, the neighborhood of $\alpha$ in $\Ga$ is $\{\beta\}$ and $L$ also fixes the vertex $\beta$.
	Label the vertices of $\Del$ as
	\[V\Del:=\{(i,\a_j)\mid 1\leqslant i\leqslant m,1\leqslant j\leqslant n\}.\]
	Then the adjacency relation `$\sim$' of $\Del$ is such that
	\[\mbox{$(i,\a_j)\sim(i',\a_{j'})$ in $\Del$ $\Longleftrightarrow\a_j\sim\a_{j'}$ in $\Ga$.}\]
	Let $R\cong\mathbb{Z}_m$ be the regular cyclic group generated by the cycle $(12\dots m)$, and let
	\[\widetilde{G}=R\wr G=R^n{:}G,\]
	with action of $\widetilde{G}$ on $V\Del$ defined by
	\[(i,a_j)^{(r_1,...,r_n;x)}=(i^{r_j},a_j^x),\mbox{ where $(r_1,...,r_n;x)\in \widetilde{G}=R^n{:}G$}.\]
	It follows that $\widetilde{G}$ is a group of automorphisms of the graph $\Del$.
	By the definition of the action of $\widetilde{G}$ on $V\Del$, we observe that the stabilizer of $\widetilde{G}$ of the vertex $(1,\a)$ is
	\[\widetilde{G}_{\alpha}=\langle (1,r_2,...,r_n)\mid r_j\in R,\ j=2,...,n\rangle{:}G_{\alpha}\cong R^{n-1}{:}L.\]
	Note that the neighborhood of $(1,\alpha)$ in $\Del$ is
	\[\Del(1,\alpha)=\left\{(i,\beta)\mid i=1,...,m\right\}.\]
	It follows that $\widetilde{G}_{\alpha}\cong R^{n-1}{:}L$ is transitive on the neighborhood of the vertex $(1,\a)$ in $\Del$.
	It implies that $\widetilde{G}$ is transitive on arcs of $\Del$, and hence
	\[\Del=\Cos(\widetilde{G},\widetilde{G}_{\alpha},\widetilde{G}_{\alpha}g\widetilde{G}_{\alpha}).\]
	
	Let $\widetilde{H}=R\wr H=R^n{:}H\leqslant\widetilde{G}$.
	Then $\widetilde{H}>\widetilde{G}_{\alpha}\cong R^{n-1}{:}L$ and $\Cos(\widetilde{G},\widetilde{H},\widetilde{H}g\widetilde{H})$ is a quotient of $\Del$ by Proposition~\ref{prop:extender}.
	Note that the normal subgroup $R^n\lhd R^n\wr H=\widetilde{H}$ is the core of $\widetilde{H}$ in $\widetilde{G}$.
	Since $\widetilde{G}/R^n=G$ and $\widetilde{H}/R^n=H$, we have
	\[\Cos(\widetilde{G},\widetilde{H},\widetilde{H}g\widetilde{H})=\Cos(G,H,HgH)=\Sig.\]
	Hence $\Sig$ is a quotient graph of $\Del$ such that $\val(\Sig)=\val(\Del)=m$.

	Since $(\alpha,\beta)$ is an arc in $\Ga$, the vertices $(1,\alpha)$ and $(1,\beta)$ are adjacent in $\Del$.
	The images of $(1,\alpha)$ and $(1,\beta)$ in $\Sig$ are
	\[\begin{aligned}
		\overline{(1,\alpha)}&=\left\{(1,\alpha)^{\widetilde{h}}\mid\widetilde{h}\in\widetilde{H}\right\}=\left\{(i,\alpha^h)\mid h\in H,\ i=1,...,m\right\}\mbox{ and }\\
		\overline{(1,\beta)}&=\left\{(1,\alpha)^{\widetilde{h}g}\mid\widetilde{h}\in\widetilde{H}\right\}=\left\{(i,\alpha^{hg})\mid h\in H,\ i=1,...,m\right\}\mbox{, respectively.}
	\end{aligned}\]
	Since $(1,\alpha)$ is adjacent to $(i,\beta)$ for all $i=1,...,m$ and $m\geqslant 2$ in $\Del$, the subgraph generated by blocks $\overline{(1,\alpha)}$ and $\overline{(1,\beta)}$ in $\Del$ is not a perfect matching.
	Therefore, $\Del$ is a disconnected symmetric pseudocover of $\Sig$.
\end{proof}

\section{Multicovers, covers and pseudocovers}

Let $\Sig$ and $\Ga$ be two $G$-symmetric graphs, and assume that $\Sig$ is connected and $\Ga$ is an extender of $\Sig$.
Let $\a$ be a vertex of $\Ga$, and let $\ov\a$ be the image of $\a$ in the quotient graph $\Sig$, namely, $\ov\a$ is a block containing $\a$.
For convenience, if $\a$ is adjacent to a vertex $\b\in\ov\b$, then we say that $\a$ is adjacent to $\ov\b$ and $\overline{\beta}$ is adjacent to $\alpha$.
We remark again that the action $G$ on the vertex set of $\Sig$ is not necessarily faithful, and $G_\a$ naturally acts on $\Sig(\ov\a)$ as a subgroup of $G_{\ov\a}$.

\begin{lemma}\label{multi-cover}
	Let $\Ga$ and $\Sig$ be $G$-symmetric graphs such that $\Ga$ is an extender of $\Sig$.
	Then the following statements are equivalent, where $\a$ is a vertex of $\Ga$:
	\begin{enumerate}[{\rm(a)}]
		\item $\Ga$ is a multicover of $\Sig$;
		\item $\a$ is adjacent to each vertex $\ov\b$ in $\Sig(\ov\a)$;
		\item $G_\a$ is transitive on $\Sig(\ov\a)$.
	\end{enumerate}
\end{lemma}
\begin{proof}
	Assume that $\Ga$ is a multicover of $\Sig$.
	Let $(\ov\a,\ov\b)$ be an arc of $\Sig$.
	Then $\Ga[\ov\a,\ov\b]$ is a regular graph, and thus each vertex of $\Ga$ in the block $\ov\a$ is adjacent to some vertices in $\ov\b$ in $\Ga$.
	So part~(a) implies part~(b).

	Now, assume that $\a$ is adjacent to each vertex $\ov\b$ in $\Sig(\ov\a)$.
	Since $G$ is arc-transitive on $\Ga$, the stabilizer $G_\a$ is transitive on the neighborhood $\Ga(\a)$, and $G_\a$ is transitive on $\Sig(\ov\a)$, namely, part ~(b) implies part~(c).

	Finally, suppose that $G_\a$ is transitive on $\Sig(\ov\a)$.
	Let $\b\in\Ga(\a)$.
	Then $\ov\b\in\Sig(\ov\a)$, and so $\ov\b^{G_\a}=\Sig(\ov\a)$.
	Thus $\a$ is adjacent to $\b^\ell$ for any $\ell\in G_\a$, which is in $\ov\b^\ell\in\Sig(\ov\a)$, namely, $\a$ is adjacent to each vertex $\ov\g\in\Sig(\ov\a)$.
	Note $G_{\overline{\alpha}}$ is transitive on the block $\overline{\alpha}$, it follows that each vertex in $\overline{\alpha}$ is adjacent to $\overline{\beta}$.
	In addition, the degree of $\alpha$ in $\Ga[\overline{\alpha},\overline{\beta}]$ is equal to $|\Ga(\alpha)\cap \overline{\beta}|$, which is independent of the choices of adjacent vertices $\alpha,\beta$ in $\Ga$.
	It follows that $\Ga[\ov\a,\ov\b]$ is a regular graph, and $\Ga$ is a multicover of $\Sig$.
	So part~(c) implies part~(a).
\end{proof}

Now we are ready to prove Theorem~\ref{thm:pscover}.

\begin{proof}[Proof of Theorem~\ref{thm:pscover}]
	Since $\val(\Sig)=\val(\Ga)$, part~(a) and part~(b) are equivalent by Lemma~\ref{multi-cover}.
	It is obvious that part~(b) and part~(c) are equivalent by Frattini's argument.

	Note that the subgroup $G_\a\leqslant G_{\ov\a}$ acts on $\Sig(\ov\a)$ naturally.
	The orbit $\ov\b^{G_\a}$ of this action has length $|G_\a:G_{\a\ov\b}|$, which is less than or equal to $|\Sig(\ov\a)|$.
	Moreover,
	\[\begin{array}{rcl}
	G_{\a\b}=G_{\a\ov\b} &\Longleftrightarrow&
	|G_\a:G_{\a\ov\b}|=|G_\a:G_{\a\b}|=\val(\Ga)=\val(\Sig)=|\Sig(\ov\a)|\\
	&\Longleftrightarrow&\mbox{$G_\a$ is transitive on $\Sig(\ov\a)$,}\\
	&\Longleftrightarrow&\mbox{$\Ga$ is a cover of $\Sig$.}
	\end{array}\]
	Additionally, $G_{\alpha\overline{\beta}}=G_{\overline{\alpha}\beta}$ if and only if $G_{\alpha\overline{\beta}}=G_{\alpha\overline{\beta}}\cap G_{\overline{\alpha}\beta}=G_{\alpha\beta}$.
	Thus either $G_{\alpha\beta}=G_{\overline{\alpha}\beta}=G_{\alpha\overline{\beta}}$ or these three groups are pairwise different.
	This proves that part~(a) and part~(d) are equivalent.
\end{proof}

In the end of this section, we prove that a pseudocover of a pseudocover of a graph is a pseudocover of the graph.

\begin{lemma}\label{lem:chain}
	Let $\Ga_i=\mathrm{Cos}\bigl(G,H_i,H_i g H_i\bigr)$ for $1\leqslant i\leqslant n$ with $3\leqslant n$, and
	assume that $\Ga_i$'s have the same valency, and $H_1<H_2<\cdots<H_n$.
	Then
	\begin{enumerate}[{\rm (a)}]
		\item $\Ga_1$ is a symmetric cover of $\Ga_n$ if and only if $\Ga_i$ is a cover of $\Ga_{i+1}$ for each $1\leqslant i<n$;
		\item $\Ga_1$ is a symmetric pseudocover of $\Ga_n$ if and only if $\Ga_i$ is a pseudocover of $\Ga_{i+1}$ for some $i$ with $1\leqslant i<n$.
	\end{enumerate}
\end{lemma}
\begin{proof}
	The two parts of the lemma are clearly equivalent, so we only need to verify part~(a).

	Denote the arc $(H_i,H_ig)$ of $\Ga_i$ by $(\alpha_i,\beta_i)$ for $1\leqslant i\leqslant n$.
	Let $1\leqslant i <j\leqslant n$, then $\Ga_j$ is a quotient of $\Ga_i$, and the quotient image of the arc $(\alpha_i,\beta_i)$ in $\Ga_{j}$ is the arc $(\alpha_{j},\beta_{j})$.
	Note that $G_{\alpha_i\beta_i}=H_i\cap H_i^g$ and $G_{\alpha_i\beta_j}=H_i\cap H_j^g$.
	Since $\val(\Ga_i)=\val(\Ga_j)$, by the equivalence of parts~(a), (c) and (d) of Theorem~\ref{thm:pscover}, $\Ga_i$ is a cover of $\Ga_{j}$ if and only if one of the following equalities hold:
	\begin{enumerate}[{\rm(1)}]
		\item $G_{\alpha_i\beta_i}=G_{\alpha_i\beta_j}\Longleftrightarrow H_i\cap H_i^g=H_i\cap H_j^g$;
		\item $G_{\alpha_{i}}=G_{\alpha_{i}}G_{\alpha_j\beta_j}\Longleftrightarrow H_{i}=H_i(H_j\cap H_j^g)$.
	\end{enumerate}

	First, we suppose that $\Ga_i$ is a cover of $\Ga_{i+1}$ for $1\leqslant i <n$.
	Then $H_i\cap H_{i+1}^g=H_i\cap H_i^g$ for $1\leqslant i <n$ and thus
	\[\begin{aligned}
		&H_1\cap H_n^g=(H_1\cap H_{n-1})\cap H_n^g=H_1\cap (H_{n-1}\cap H_n^g)=H_1\cap (H_{n-1}\cap H_{n-1}^g)\\
		=&H_1\cap H_{n-1}^g=(H_1\cap H_{n-2})\cap H_{n-1}^g=H_1\cap (H_{n-2}\cap H_{n-1}^g)=H_1\cap (H_{n-2}\cap H_{n-2}^g)\\
		=&\cdots\\
		=&H_1\cap H_2^g=H_1\cap H_1^g.
	\end{aligned}\]
	Since $H_1\cap H_n^g=H_1\cap H_1^g$ and $\val( \Ga_1)=\val( \Ga_n)$, we obtain that $\Ga_1$ is a cover of $\Ga_n$.

	On the other hand, suppose that $\Ga_1$ is a cover of $\Ga_n$.
	Then $H_1\cap H_{n}^g=H_1\cap H_1^g$.
	Since $H_1\cap H_1^g\leqslant H_1\cap H_{n-1}^g\leqslant H_1\cap H_{n}^g=H_1\cap H_1^g$, we have $H_1\cap H_{n-1}^g=H_1\cap H_1^g$, and thus $\Ga_1$ is a cover of $\Ga_{n-1}$.
	In addition, we also have the equality $H_n=H_1(H_n\cap H_n^g)$, and then
	\[H_n=H_n(H_n\cap H_n^g)\geqslant H_{n-1}(H_n\cap H_n^g)\geqslant H_1(H_n\cap H_n^g)=H_n.\]
	Hence $H_n=H_{n-1}(H_n\cap H_n^g)$, and therefore $\Ga_n$ is a cover of $\Ga_{n-1}$.

	Applying the above process with the condition that $\Ga_1$ is a cover of $\Ga_{n-1}$, we obtain that $\Ga_1$ is a cover of $\Ga_{n-2}$ and $\Ga_{n-2}$ is a cover of $\Ga_{n-1}$.
	Therefore, by repeating the above process, we can deduce that $\Ga_i$ is a cover of $\Ga_{i+1}$ for each $1\leqslant i <n$.
\end{proof}

\section{Praeger-Xu's graphs}

In this section, we apply Theorem~\ref{thm:pscover} to characterize Praeger-Xu's graphs in the language of covers and pseudocovers.

Praeger and Xu~\cite{praeger1989characterization} characterized a class of symmetric graphs by proving that if a connected $G$-symmetric graph $\Ga$ of valency $2p$ with $p$ prime such that $G$ has an abelian normal $p$-subgroup $M$ which is not vertex-semiregular, then $\Ga$ belongs to an explicit class of symmetric graphs.
In terminology of coset graphs, Praeger-Xu's graphs can be defined as follows.

\begin{definition}\label{def:Praeger-Xu}
	Let $p$ be a prime and $s,r$ integers such that $r\geqslant 3$ and $1\leqslant s<r$.
	Let $G=M{:}T\cong\ZZ_p^r{:}\D_{2r}$, where
	$M=\l a_0\r\times\dots\times\l a_{r-1}\r\cong\ZZ_p^r$, and $T=\langle x,y\mid x^r=y^2=1,x^y=x^{-1}\rangle\cong \mathrm{D}_{2r}$ such that $a_i^x=a_{i+1}$ and $a_i^y=a_{r-1-i}$ for $i=0,...,r-1$ (reading $i+1\pmod{r}$).
	Define a subgroup $H_s$ and an involution $g_s$ for each value $s\in\{1,2,\dots,r\}$:
	\begin{enumerate}[(1)]
		\item for $s$ even, let $H_s=\langle a_i\mid \frac{s}{2}\leqslant i <r-\frac{s}{2}\rangle{:}\langle y\rangle$, and $g_s=xy$;
		\item for $s$ odd, let $H_s=\langle a_i\mid \frac{s-1}{2}\leqslant i < r-\frac{s+1}{2}\rangle{:}\langle xy\rangle$, and $g_s=y$.
	\end{enumerate}
	Set $C(p,r,s)=\Cos(G,H_s,H_sg_sH_s)$.
\end{definition}

We remark that the graph $C(p,r,s)$ is the Praeger-Xu's graph with the same notation in \cite{praeger1989characterization} and $C(p,r,1)$ is isomorphic to the wreath graph $\mathbf{W}(r,p)=\C_r[\ov\K_p]$.
The following theorem gives a characterization of Praeger-Xu's graphs in the language of covers and pseudocovers.

\begin{theorem}\label{thm:PX-graphs}
	With the notations defined above, the followings hold.
	\begin{enumerate}[{\rm(a)}]
		\item $C(p,r,s)$ is a connected $G$-symmetric graph of order $rp^s$ and valency $2p$. The elementary abelian normal $p$-subgroup $M\lhd G$ is not semiregular on vertices;
		\item $C(p,r,s)$ is a $G$-symmetric pseudocover of $C(p,r,s-2)$, for $3\leqslant s< r$;
		\item $C(p,r,2)$ is a $G$-symmetric cover of $\mathbf{W}(r,p)$
	\end{enumerate}
\end{theorem}
\begin{proof}
	\begin{enumerate}[{\rm (a)}]
		\item Obviously, $C(p,r,s)$ is a connected $G$-symmetric graph. Since $H_s\cong\mathbb{Z}_{p}^{r-s}{:}\mathbb{Z}_2$, the graph $C(p,r,s)$ has order $rp^{s}$.
		The abelian normal $p$-subgroup $M\lhd G$ is not semiregular on vertices since $M\cap H_s\cong \mathbb{Z}_{p}^{r-s}\neq 1$.
		
		When $s$ is even, we obtain that
		\[\begin{aligned}
			&H_s^{g_s}=H_s^{xy}=\left\langle a_i\mid  \frac{s}{2}-2< i \leqslant r-\frac{s}{2}-2\right\rangle{:}\langle x^2y\rangle\mbox{ and }\\
			&H_s\cap H_s^{g_s}=\left\langle a_i\mid\frac{s}{2}\leqslant i<r-\frac{s}{2}-1\right\rangle\cong \mathbb{Z}_p^{r-s-1}.
		\end{aligned}\]
		Hence $H_s\cap H_s^g\cong\mathbb{Z}_p^{r-s-1}$ and $C(p,r,s)$ has valency equal to $|H_s:H_s\cap H_s^{g_s}|=2p$.

		When $s$ is odd, similarly the graph $C(p,r,s)$ has valency $2p$ since
		\[\begin{aligned}
			&H_s^{g_s}=H_s^y=\left\langle a_i\mid \frac{s-1}{2}< i \leqslant r-\frac{s+1}{2}\right\rangle{:}\langle yx\rangle\mbox{ and }\\
			&H_s\cap H_s^{g_s}=\left\langle a_i\mid\frac{s+1}{2}\leqslant i<r-\frac{s+1}{2}\right\rangle\cong \mathbb{Z}_p^{r-s-1}.
		\end{aligned}\]
		We complete the proof of the part~(a).

		\item Note that $H_{s-2}<H_s$ for $3\leqslant s< r$, $C(p,r,s)$ is a $G$-symmetric extender of $C(p,r,s-2)$ for $3\leqslant s< r$. Recall the equivalence between part~(a) and (d) of Theorem~\ref{thm:pscover}, since each graph is of valency $4$, $C(p,r,s)$ is a pseudocover of $C(p,r,s-2)$ if and only if $H_{s}\cap H_{s-2}^{g_s}\neq H_{s}\cap H_s^{g_s}$.
		
		When $s$ is even, we notice that $a_{r-\frac{s}{2}-1}\in H_s$ by the definition and $a_{r-\frac{s}{2}-1}\in H_{s-2}\cap H_{s-2}^{g_s}$ given in part~(a).
		Thus $a_{r-\frac{s}{2}-1}\in H_{s}\cap H_{s-2}^{g_s}$.
		However, $a_{r-\frac{s}{2}-1}\notin H_{s}\cap H_{s}^{g_s}$ by part~(a), namely $H_{s}\cap H_{s-2}^{g_s}\neq H_{s}\cap H_s^{g_s}$.
		Therefore, $C(p,r,s)$ is a pseudocover of $C(p,r,s-2)$.
		
		When $s$ is odd, for the similar reasons that $a_{{\frac{s+1}{2}}}$ is the element in $H_{s}\cap H_{s-2}^{g_s}$ but not in $H_s\cap H_{s-2}^{g_s}$.
		Thus $H_{s}\cap H_{s-2}^{g_s}\neq H_{s}\cap H_s^{g_s}$, and therefore $C(p,r,s)$ is a pseudocover of $C(p,r,s-2)$ for all $3\leqslant s<r$, as in part~(b).

		\item Let $H=H_2=\langle a_1,...,a_{r-2}\rangle{:}\langle y\rangle$ and $g=g_2=xy$. Set
		\[T=\bigl(\langle a_0a_{r-1}\rangle\times\langle a_1,...,a_{r-2}\rangle\bigr){:}\langle y\rangle>H_2\]
		then $T\cong\mathbb{Z}_p^{r-1}{:}\mathbb{Z}_2$ and
		\[T^{g}=\bigl(\langle a_{r-2}a_{r-1}\rangle\times\langle a_{r-3},a_{r-4}...,a_0\rangle\bigr){:}\langle y^{xy}\rangle.\]
		Thus
		\[T\cap T^{g}=\langle a_0a_{r-1}a_{r-2}\rangle\times\langle a_1,...,a_{r-3}\rangle\cong\mathbb{Z}_p^{r-2}.\]
		Hence $\Sig=\mathrm{Cos}\bigl(G,T,T g T\bigr)$ has order $rp$ and valency $2p$.
		Furthermore, $M\lhd G$ is not semiregular on $V\Sig$ since $M\cap T\neq 1$.
		Then $\Sig$ is isomorphic to $C(p,r,1)\cong \mathbf{W}(r,p)$.
		Moreover, $H\cap T^{g}=\langle a_1\cdots, a_{r-3}\rangle= H\cap H^{g}$.
		Therefore, $C(p,r,2)=\mathrm{Cos}\bigl(G,H,H g H\bigr)$ is a cover of $\Sig\cong C(p,r,1)=\mathbf{W}(r,p)$ by the equivalence of part~(a) and (d) of Theorem~\ref{thm:pscover}.
	\end{enumerate}
We complete the proof.
\end{proof}

\begin{proof}[Proof of Corollary~$\ref{cor:PX-graphs}$:]
	The corollary follows from Theorem~\ref{thm:PX-graphs} and Lemma~\ref{lem:chain} with the single exception given in the part~(c) of Theorem~\ref{thm:PX-graphs}.
\end{proof}

\section{Pseudocovers of tetravalent graphs}\label{sec:psval4}

Let $\Sig$ be a connected $G$-symmetric tetravalent graph with $G\leqslant \Aut\Sig$ and $(\omega,\omega')$ an arc of $\Sig$.
Then $\Sig=\mathrm{Cos}\bigl(G,G_\omega,G_\omega g G_\omega\bigr)$ where $g\in G$ such that $(\omega,\omega')^g=(\omega',\omega)$ and the arc stabilizer of $G$ is $G_{\o\o'}=G_\o\cap G_{\o'}=G_\o\cap G_\o^g$.

As usual, let $G_\omega^{[1]}$ be the core of $G_\omega$ acting on $\Sig(\omega)$, which is the pointwise stabilizer of $G_\omega$ on $\Sig(\omega)$.
Then $G_\o^{[1]}\lhd G_{\o}$, and the permutation group $G_\o^{\Sig(\o)}$ induced by $G_\o$ on $\Sig(\o)$ is such that
\[G_\o^{\Sig(\o)}\cong G_\o/G_\o^{[1]}.\]
Since $\Sig$ is tetravalent and $G$-symmetric, $G_\o^{\Sig(\o)}$ is a transitive permutation group of degree $4$.

Observe that $G_\o^{[1]},G_{\o'}^{[1]}\lhd G_{\o\o'}$, and
\[G_{\o\o'}^{\Sig(\o)}\cong G_{\o\o'}/G_{\o}^{[1]}\cong G_{\o\o'}/G_{\o'}^{[1]}\cong G_{\o\o'}^{\Sig(\o')}.\]
The intersection $G_\o^{[1]}\cap G_{\o'}^{[1]}$ fixes each vertex in $\Sig(\o)\cup\Sig(\o')$, which is denoted by $G_{\o\o'}^{[1]}$.
Obviously, $G_{\o\o'}^{[1]}\lhd G_{\o\o'}$, and
\[G_{\o\o'}/G_{\o\o'}^{[1]}\cong G_{\o\o'}^{\Sig(\o)\cup\Sig(\o')}\]
is the permutation group on the set $\Sig(\o)\cup\Sig(\o')$ induced by the arc stabilizer $G_{\o\o'}$.

Poto{\v c}nik characterized the vertex stabilizer of $s$-arc-transitive tetravalent graphs $\Sig$ of $s\geqslant 2$ (see~\cite{potocnik2009list}).
The sizes of vertex stabilizers of such graphs are bounded by $2^43^6$.
If $\Ga$ is a connected $G$-symmetric pseudocover of $\Sig$ with $G\leqslant\Aut\Sig$, then $\Ga$ is not $(G,s)$-arc-transitive for $s\geqslant 2$ by Corollary~\ref{coro:locallyprim}.
Thus the size of vertex stabilizer of $G$ on $\Ga$ is small and no more than $16$.
Hence we focus on the case that $\Sig$ is not $(G,2)$-arc-transitive in this section.
Then $G_\o$ is a $2$-group and there exist infinite families of such graphs with $|G_\o|$ unbounded, Praeger-Xu's graphs $C(2,r,s)$ for instance.

\begin{lemma}\label{lem:tetra-1}
	If $G_\o$ is a $2$-group, then either $G_\o^{\Sig(\o)}\cong\D_8$ and $G_{\o\o'}^{\Sig(\o)}\cong\ZZ_2$, or $G_\o\cong\ZZ_2^2$ or $\ZZ_4$.
\end{lemma}
\begin{proof}
	Since $G_\o$ is a $2$-group, the induced permutation group $G_\omega^{\Sig(\omega)}\cong G_\omega/G_\omega^{[1]}$ is a transitive $2$-group on $4$ elements. Hence $G_\omega^{\Sig(\omega)}\cong \D_8$, $\mathbb{Z}_2^2$ or $\mathbb{Z}_4$.
	
	Assume that $G_\omega^{\Sig(\omega)}\cong \D_8$.
	Then $G_{\omega\omega'}^{\Sig(\omega)}\cong\mathbb{Z}_2$ is the stabilizer of $G_\omega^{\Sig(\omega)}$ on $4$ elements.
	
	Assume that $G_\omega^{\Sig(\omega)}\cong\mathbb{Z}_2^2$ or $\mathbb{Z}_4$.
	Then $\left|G_\omega^{\Sig(\omega)}\right|=4$ and $G_{\omega\omega'}^{\Sig(\omega)}=1$.
	Hence $G_{\omega\omega'}=G_\omega^{[1]}$ since $G_{\omega\omega'}^{\Sig(\omega)}=G_{\omega\omega'}/G_{\omega}^{[1]}$.
	It follows from $(\omega,\omega')^g=(\omega',\omega)$ that $G_{\omega}^{[1]}=G_{\omega\omega'}= G_{\omega\omega'}^g=\left(G_{\omega}^{[1]}\right)^g=G_{\omega'}^{[1]}$, and thus $G_{\omega}^{[1]}=G_{\omega'}^{[1]}$.
	Since $\Sig$ is a connected $G$-symmetric graph and $(\omega,\omega')$ is an arc of $\Sig$, this implies that $G_{\omega}^{[1]}=G_{\omega_0}^{[1]}$ for any $\omega_0\in V\Sig$.
	Thus $G_{\omega}^{[1]}$ fixes all the vertices of $\Sig$, and hence $G_{\omega}^{[1]}=1$.
	Therefore, $G_\omega\cong G_{\omega}^{\Sig(\omega)}\cong \mathbb{Z}_2^2$ or $\mathbb{Z}_4$.
\end{proof}

\begin{lemma}\label{lem:tetra-2}
	If $G_\o$ is a $2$-group of order less than $16$, then $G_\o\cong G_\o^{\Sig(\o)}\cong\D_8$, $\mathbb{Z}_2^2$ or $\mathbb{Z}_4$, and $\Sig$ does not have a $G$-symmetric pseudocover.
\end{lemma}
\begin{proof}
	Since $\Sig$ is $G$-symmetric with valency $4$, the order of $G_\o$ is divisible by $4$.
	Hence $|G_\o|=4$ or $8$.
	By Lemma~\ref{lem:tetra-1} $G_\o\cong \mathbb{Z}_2^2$ or $\mathbb{Z}_4$ if $|G_\o|=4$, and $G_{\o}^{\Sig(\omega)}\cong \D_8$ if $|G_\o|=8$.
	When $|G_\o|=8$, the kernel $G_{\omega}^{[1]}$ has size $|G_{\omega}^{[1]}|=|G_\o|/|G_{\o}^{\Sig(\omega)}|=1$.
	Thus $G_\omega\cong G_{\o}^{\Sig(\omega)}/G_\o^{[1]}\cong \D_8$.

	Suppose $\Ga$ is a connected $G$-symmetric pseudocover and $(\alpha,\alpha')$ is an arc in $\Ga$ such that $(\omega,\omega')=(\overline{\alpha},\overline{\alpha'})$.
	Then $G_{\alpha\alpha'}<G_{\alpha\omega'}< G_{\omega\omega'}$ by part~(d) of Theorem~\ref{thm:pscover}.
	Thus $|G_{\omega\omega'}|/|G_{\alpha\alpha'}|\geqslant 4$, $|G_{\omega\omega'}|\geqslant 4$ and $|G_\o|\geqslant 4\times 4=16$.
	Hence, $\Sig$ has no connected $G$-symmetric pseudocover when $|G_\o|<16$.
\end{proof}

Now we start to construct connected $G$-symmetric pseudocovers of a tetravalent graph $\Sig=\mathrm{Cos}\bigl(G,G_\omega,G_\omega g G_\omega\bigr)$, where $G\leqslant\Aut\Sig$ and $G_\o$ is a $2$-group of size divisible by $16$.

\begin{lemma}\label{lem:tetra-3}
	Let $G_\o$ be a $2$-group of order at least $16$.
	Then the following statements hold:
	\begin{enumerate}[{\rm(a)}]
		\item $G_{\o\o'}/G_{\o\o'}^{[1]}\cong G_{\o\o'}^{\Sig(\o)\cup\Sig(\o')}\cong G_{\o\o'}^{\Sig(\o)}\times G_{\o\o'}^{\Sig(\o')}\cong \ZZ_2\times\ZZ_2$;
		\item there exist elements $x\in G_\o\setminus G_{\o\o'}$ and $y\in G_{\omega\omega'}\setminus G_{\omega}^{[1]}$ such that
		\[G_\o=\l G_{\o\o'},x\r=\l G_\o^{[1]},x,y\r\ \  \text{ and }\ \ G_{\omega}^{[1]}=\langle G_{\omega\omega'}^{[1]},y^g\rangle,\]
		and further, $G_{\o\o'}/G_\o^{[1]}=\l\ov y\r\cong\ZZ_2$, and $G_\o/G_\o^{[1]}=\l\ov x\r{:}\l\ov y\r\cong\D_8$;
		\item there exists $\widetilde{g}\in G_\o gG_\o$ such that $G_\o gG_\o=G_\o \widetilde{g} G_\o$  and $\widetilde{g}^2\in G_\o^{[1]}$.
	\end{enumerate}
\end{lemma}
\begin{proof}
	\begin{enumerate}[{\rm(a)}]
		\item By Lemma~\ref{lem:tetra-1}, we have $G_\o/G_\o^{[1]}\cong\D_8$ and $G_{\o\o'}/G_\o^{[1]}\cong\ZZ_2$.
		Since $G_\o^{[1]}\lhd G_{\o\o'}$, we obtain that
		\[G_{\o'}^{[1]}=(G_\o^{[1]})^g\lhd G_{\o\o'}^g=G_{\o\o'}\mbox{ and }|G_{\o\o'}:G_{\o'}^{[1]}|=|G_{\o\o'}:G_\o^{[1]}|=2.\]
		Note that $G_{\o'}^{[1]}\not=G_\o^{[1]}$ since $\Sig$ is connected and $(\o,\o')$ is an arc.
		Hence $G_{\o\o'}^{[1]}<G_{\o}^{[1]}\cong G_{\o'}^{[1]}$.
		Therefore, $|G_{\omega\omega'}:G_{\omega\omega'}^{[1]}|=4$ and $G_{\o\o'}/G_{\o\o'}^{[1]}\cong G_{\o\o'}^{\Sig(\o)\cup\Sig(\o')}\leqslant G_{\o\o'}^{\Sig(\o)}\times G_{\o\o'}^{\Sig(\o')}=\ZZ_2\times\ZZ_2$.

		\item Recall that $G_{\o\o'}^{[1]}\lhd G_{\omega'}^{[1]}\lhd G_{\o\o'}$ such that $|G_{\o\o'}:G_\o^{[1]}|=|G_\o^{[1]}:G_{\o\o'}^{[1]}|=|G_{\o'}^{[1]}: G_{\omega\omega'}^{[1]}|=2$ by part~(a).
		Let $y\in G_{\o'}^{[1]}\setminus G_{\omega\omega'}^{[1]}$.
		Then $y^2\in G_{\omega\omega'}^{[1]}\leqslant G_\o^{[1]}$, and
		\[G_{\omega'}^{[1]}=\langle G_{\omega\omega'}^{[1]},y\rangle,\ \ \ G_{\omega}^{[1]}=\langle (G_{\omega\omega'}^{[1]})^g,y^g\rangle=\langle G_{\omega\omega'}^{[1]},y^g\rangle.\]
		Observe that $G_{\omega\omega'}=G_\o^{[1]} G_{\o'}^{[1]}=\langle G_{\o\o'}^{[1]},y,y^g\rangle$ and $G_{\omega\omega'}/G_{\omega\omega'}^{[1]}\cong\mathbb{Z}_2^2$ by part~(a).
		Since $y^2$ and $(y^2)^g$ are both in $G_{\omega\omega'}^{[1]}$, it follows that $y\in G_{\omega\omega'}\setminus G_{\omega}^{[1]}$ and $G_{\omega\omega'}/G_{\omega}^{[1]}=\langle \overline{y}\rangle$.
		Since $G_{\o\o'}$ is the stabilizer of $G_{\o}$ acting on $\Sig(\omega)$ and $G_{\o}/G_\o^{[1]}\cong\mathrm{D}_8$, there exists $x\in G_\o$ such that $G_{\o}/G_\o^{[1]}=\langle \overline{x}\rangle{:}\langle \overline{y}\rangle\cong \mathrm{D}_8$.
		\item If $g^2\in G_\o^{[1]}$, then part~(c) holds with $\widetilde{g}=g$.
		We thus suppose that $g^2\notin G_\o^{[1]}$.
		Since $g^2\in G_{\o\o'}=\l G_\o^{[1]},y\r$, we have that $g^2=zy$ for some $z\in G_\o^{[1]}$.
		Let $\widetilde{g}=gy\in G_\o gG_\o$, then $G_\o gG_\o=G_\o \widetilde{g} G_\o$.
		Recall that $G_{\omega}^{[1]}=\langle G_{\omega\omega'}^{[1]},y^g\rangle\lhd G_\omega$ in part~(b), then $(y^g)^y\in G_\o^{[1]}$. Therefore, we have $\widetilde{g}^2=g^2y^gy=zyy^gy=zy^2y^{gy}\in G_\o^{[1]}$ as $z,y^2,y^{gy}\in G_\o^{[1]}$.
	\end{enumerate}
	We complete the proof.
\end{proof}

Now we are ready to construct pseudocovers of $\Sig$.

\begin{construction}\label{cons:pcoverval4}
	Using the notations in Lemma~\ref{lem:tetra-3}, let $g=\widetilde{g}$ and
	\[L=\l G_\o^{[1]},xy\r < \l G_\o^{[1]},x,y\r = G_\o.\]
	Define $\Ga=\mathrm{Cos}\bigl(G,L,L g L\bigr)$.
\end{construction}
\begin{figure}[!ht]
	\centering
	\includegraphics[scale=0.9]{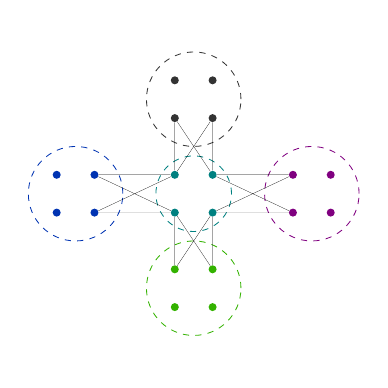}
	\caption{Neighborhoods of a block in Construction~\ref{cons:pcoverval4}}\label{fig:4cover}
\end{figure}
\begin{lemma}\label{lem:tetra-4}
	The graph $\Ga$ defined in Construction~\ref{cons:pcoverval4} is a connected $G$-symmetric pseudocover of $\Sig=\Cos(G,G_\omega,G_\omega gG_\omega)$.
\end{lemma}
\begin{proof}
	First, the graph $\Ga$ is connected since
	\[\begin{array}{rcl}
	\l L,g\r&=&\l G_\o^{[1]},xy,g\r\\
	&=& \l G_\o^{[1]},G_{\o'}^{[1]},xy,g\r, \ \mbox{as $(G_\o^{[1]})^g=G_{\o'}^{[1]}$}\\
	&=& \l G_\o^{[1]}, y^g,xy,g\r,\ \mbox{as $G_{\o'}^{[1]}=\langle G_{\omega\omega'}^{[1]},y^g \rangle$}\\
	&=& \l G_\o^{[1]},x,y,g\r\\
	&=&\l G_\o,g\r,\ \mbox{as $G_{\o}=\l G_\o^{[1]},x,y\r$}\\
	&=&G.
	\end{array}\]

	Next, noticing that $xy\notin G_{\o\o'}$ and $(xy)^2\in G_\o^{[1]}$, we have that $L\cap G_{\o\o'}=\l G_\o^{[1]},xy\r\cap G_{\o\o'}=G_\o^{[1]}$, and so $L^g\cap G_{\o\o'}=G_{\o'}^{[1]}$.
	Since $L<G_\o$, we have $L^g<G_\o^g=G_{\o'}$, and so
	\[L\cap L^g = (L\cap L^g)\cap G_{\o\o'}=G_\o^{[1]}\cap G_{\o'}^{[1]}=G_{\o\o'}^{[1]}.\]
	Therefore, recalling that $|G_\o^{[1]}:G_{\o\o'}^{[1]}|=2$, we have
	\[|L:L\cap L^g|=|L:G_{\o\o'}^{[1]}|=|L:G_\o^{[1]}||G_\o^{[1]}:G_{\o\o'}^{[1]}|=4,\]
	and so $\Ga=\Cos(G,L,LgL)$ is of valency equal to $|L:L\cap L^g|=4$.

	Finally, since $(xy)^2=y^2\in G_\o^{[1]}$, we have $L^{\Sig(\o)}\cong L/G_\o^{[1]}\cong\ZZ_2$.
	Thus $L$ is intransitive on $\Sig(\o)$, and $\Ga$ is a pseudocover of $\Sig$ by the equivalence between part~(a) and (b) of Theorem~\ref{thm:pscover}.
\end{proof}

Use the notations in Construction~\ref{cons:pcoverval4}, we remark that $L/G_\o^{[1]}=\langle\bar{x}\bar{y}\rangle\leqslant \D_8\cong G_\o^{\Sig(\o)}$ is isomorphic to $\mathbb{Z}_2$.
There are exactly $5$ subgroups of order $2$ in $G_\o^{\Sig(\o)}=\langle\overline{x},\overline{y}\rangle\cong\D_8$.
Let
\[L_1=\langle G_\o^{[1]},x^3y\rangle, L_2=\langle G_\o^{[1]},x^2y\rangle, L_3=\langle G_\o^{[1]},y\rangle=G_{\o\o'}\mbox{ and }L_4=\langle G_\o^{[1]},x^2\rangle.\]
We can define graphs $\Ga_i=\mathrm{Cos}(G,L_i,L_igL_i)$ for $i=1,2,3,4$. By conjugating $x$ on $G$, we have $\Ga_1\cong\Ga$ and $\Ga_2\cong\Ga_3$.
Note that $\Ga_3$ is disconnected since $\langle L_3,g\rangle=\langle G_{\o\o'},g\rangle\neq G$, and so is $\Ga_2$.
Furthermore, since $x^2$ and $g$ both normalize $\langle G_\o^{[1]},y\rangle=G_{\o\o'}$, it follows that $\langle L_4,g\rangle=\langle G_\o^{[1]},x^2,g\rangle\leqslant \langle G_\o^{[1]},y,x^2,g\rangle \leqslant N_G(G_{\omega\o'})\neq G$, and hence $\Ga_4$ is also disconnected.

\begin{lemma}\label{lem:tetra-5}
	Let $|G_\o|=2^s\geqslant16$.
	\begin{enumerate}[{\rm(a)}]
		\item $\Ga$ has a connected pseudocover of order $|G|/4$ if $s$ is even, and
		\item $\Ga$ has a connected pseudocover of order $|G|/8$ if $s$ is odd.
	\end{enumerate}
\end{lemma}
\begin{proof}
	By definition, the index $|G_\o:L|$ is equal to
	\[\frac{|G_\o|}{|L|}=\frac{|G_\o/G_\o^{[1]}|}{|L/G_\o^{[1]}|}=\frac{|\l\ov x\r{:}\l\ov y\r|}{|\l\ov{x}\,\ov{y}\r|}=\frac{8}{2}=4.\]
	Thus $|V\Ga|=|G:L|=|G:G_\o||G_\o:L|=4|G:G_\o|=4|V\Sig|$.
	The proof of the lemma then follows from Lemma~\ref{lem:tetra-4} and Lemma~\ref{lem:chain}.
\end{proof}

Recall that the sizes of vertex stabilizers of $s$-arc-transitive groups with $s\geqslant 2$ on tetravalent graphs divide $2^43^6$ (see~\cite{potocnik2009list}).
Thus the theorem implies that each connected tetravalent symmetric graph with vertex stabilizer of size divisible by $32$ has connected pseudocovers.

\begin{proof}[Proof of Theorem~$\ref{thm:pscover4}$]
	It is known that if the vertex stabilizer is of order divisible by $32$, then the stabilizer is a $2$-group.
	The proof of the theorem consists of Lemmas~\ref{lem:tetra-4} and \ref{lem:tetra-5}.
\end{proof}

We remark that the pseudocover $\Ga$ constructed in Construction~\ref{cons:pcoverval4} is exactly the $\mathrm{Pl}^2(\Sig)$, where $\mathrm{Pl}$ is the operator of symmetric \textit{partial line graph} introduced in~\cite{marusic2001Partial,potocnik2007Tetravalent}.

\begin{example}\label{exam:li4pscover}
	{\rm
	Suppose that $\Sig=\mathrm{Cos}\bigl(G,H,H g H\bigr)$ is a connected tetravalent graph with $H=\langle r,s\mid a^8=b^2=1,a^b=a^{-1}\rangle\cong \mathrm{D}_{16}$, $g^2=1$ and $H\cap H^g=\langle a^4,b\rangle\cong \mathbb{Z}_2^2$.
	Let $L=\langle a^4,ab\rangle\cong \mathbb{Z}_2^2$ then $\Ga=\mathrm{Cos}\bigl(G,L,L g L\bigr)$ is a connected pseudocover of $\Sig$.

	Li, Praeger and Zhou constructed such pairs $(\Ga,\Sig)$ by letting $G=\mathrm{PSL}(2,p)$ with $p\equiv 1\pmod{16}$ in~\cite[Construction 3.5]{li2010Imprimitive}.
	\qed
	}
\end{example}

\begin{example}
	{\rm
	Let $\Sig=C(2,r,s)$ be the Praeger-Xu's graph of valency $4$ with $s\geqslant 3$ and $r\geqslant s+1$.
	The pseudocover $\Ga$ of $\Sig$ constructed as in Construction~\ref{cons:pcoverval4} is exactly the graph $C(2,r,s-2)$, which has been already discussed in Theorem~\ref{thm:PX-graphs}.
	\qed
	}
\end{example}

\bigskip

\noindent{\bf Acknowledgements}
This work was partially supported by NNSFC grant no.~11931005.

\medskip
\noindent{\bf Data availability}
This manuscript has no associated data.

\section*{Declarations}

\noindent{\bf Conflict of interest} The authors declare that they have no conflicts of interest.


\medskip


\begin{thebibliography}{10}
	\bibitem{biggs1974Algebraic}
	N. Biggs.
	\newblock {\em Algebraic Graph Theory}.
	\newblock {Cambridge University Press, London}, 1974.

	\bibitem{brouwer1989Distanceregular}
	A.~E. Brouwer, A.~M. Cohen, and A.~Neumaier.
	\newblock {\em Distance-Regular Graphs}, 18 of {\em Ergebnisse der
	Mathematik und ihrer Grenzgebiete (3) [Results in Mathematics and Related
	Areas (3)]}.
	\newblock Springer-Verlag, Berlin, 1989.

	\bibitem{du1998arctransitive}
	S.~F.~Du, D.~Maru\v{s}i\v{c}, and A.~O. Waller.
	\newblock On $2$-arc-transitive covers of complete graphs.
	\newblock {\em J. Combin. Theory Ser. B}, 74(2):276--290,
	1998.

	\bibitem{du20052arctransitive}
	S.~F.~Du, J.~H.~Kwak, and M.~Y. Xu.
	\newblock 2-arc-transitive regular covers of complete graphs having the
	covering transformation group $\mathbb{Z}^3_p$.
	\newblock {\em J. Combin. Theory Ser. B}, 93(1):73--93, 2005.

	\bibitem{godsil2001Algebraic}
	C.~Godsil and G.~Royle.
	\newblock {\em Algebraic Graph Theory}, 207 of {\em Graduate Texts in
	Mathematics}.
	\newblock Springer-Verlag, New York, 2001.

	\bibitem{li2001finite}
	C.~H.~Li.
	\newblock The finite vertex-primitive and vertex-biprimitive $s$-transitive
	graphs for $s \geqslant 4$.
	\newblock {\em Trans. Amer. Math. Soc.},
	353(9):3511--3529, 2001.


	\bibitem{li2010Imprimitive}
	C.~H.~Li, C.~E.~Praeger, and S.~Zhou.
	\newblock Imprimitive symmetric graphs with cyclic blocks.
	\newblock {\em European J. Combin.}, 31(1):362--367, 2010.

	\bibitem{liebeck1988NanScott}
	M.~W.~Liebeck, C.~E.~Praeger, and J.~Saxl.
	\newblock On the {O}'{Nan-Scott} theorem for finite primitive permutation
	groups.
	\newblock {\em J. Austral. Math. Soc. Ser. A}, 44(3):389--396, 1988.
	
	\bibitem{marusic2001Partial}
	D.~Maru\v{s}i\v{c} and R.~Nedela.
	\newblock Partial line graph operator and half-arc-transitive group actions.
	\newblock {\em Math. Slovaca}, 51(3):241--257, 2001.
	
	\bibitem{potocnik2009list}
	P.~Poto\v{c}nik.
	\newblock A list of $4$-valent $2$-arc-transitive graphs and finite faithful
	amalgams of index $(4,2)$.
	\newblock {\em European J. Combin.}, 30(5):1323--1336, 2009.

	\bibitem{potocnik2007Tetravalent}
	P.~Poto{\v c}nik and S.~Wilson.
	\newblock Tetravalent edge-transitive graphs of girth at most 4.
	\newblock {\em J. Combin. Theory Ser. B}, 97(2):217--236,
	2007.
	
	\bibitem{praeger1993NanScott}
	C.~E.~Praeger.
	\newblock An {O}'{Nan-Scott} theorem for finite quasiprimitive permutation
	groups and an application to $2$-arc transitive graphs.
	\newblock {\em J. Lond. Math. Soc. (2)}, 47(2):227--239, 1993.

	\bibitem{praeger1989characterization}
	C.~E. Praeger and M.~Y. Xu.
	\newblock A characterization of a class of symmetric graphs of twice prime
	valency.
	\newblock {\em European J. Combin.}, 10(1):91--102, 1989.
	
	
\end{thebibliography}
\end{document}